\documentclass[sn-mathphys,Numbered]{sn-jnl}

\usepackage{graphicx}%
\usepackage{multirow}%
\usepackage{amsmath,amssymb,amsfonts}%
\usepackage{amsthm}%
\usepackage{mathrsfs}%
\usepackage[title]{appendix}%
\usepackage{xcolor}%
\usepackage{textcomp}%
\usepackage{manyfoot}%
\usepackage{booktabs}%
\usepackage{algorithm}%
\usepackage{algorithmicx}%
\usepackage{algpseudocode}%
\usepackage{listings}%
\usepackage{amsmath, amssymb, amsthm}
\usepackage{ dsfont } 
\usepackage{csquotes} 
\usepackage{todonotes}
\usepackage{bbm}
\usepackage{float}
\usepackage[utf8]{inputenc}

\newtheorem{theorem}{Theorem}[section]
\newtheorem{definition}[theorem]{Definition}

\newtheorem{proposition}[theorem]{Proposition}
\newtheorem{lemma}[theorem]{Lemma}
\newtheorem{corollary}[theorem]{Corollary}

\newcommand{\R}{\mathds{R}}

\newcommand{\N}{\mathds{N}}
\newcommand{\1}{\mathbbm{1}}

\DeclareMathOperator{\dark}{Dark}
\DeclareMathOperator{\argmin}{argmin}
\DeclareMathOperator{\conv}{conv}

\DeclareMathOperator{\cone}{cone}

\DeclareMathOperator{\inter}{int}

\newcommand{\eps}{\varepsilon}
\newcommand{\multibinom}[2]{\left[\begin{array}{cc}
		#1 \\
		#2
	\end{array} \right]}

\newcounter{problems} 

\begin{document}

\title[Bounds on polarization problems on compact sets via MIP]{Bounds on polarization problems on compact sets via mixed integer programming}


\author*[1]{\fnm{Jan} \sur{Rolfes}}\email{jrolfes@kth.se}
\equalcont{These authors contributed equally to this work.}

\author[2]{\fnm{Robert} \sur{Sch\"uler}}\email{robert.schueler2@uni-rostock.de}
\equalcont{These authors contributed equally to this work.}

\author[3]{\fnm{Marc Christian} \sur{Zimmermann}}\email{marc.christian.zimmermann@gmail.com}
\equalcont{These authors contributed equally to this work.}

\affil*[1]{\orgdiv{Optimization and System Theory}, \orgname{KTH - Royal Institute of Technology}, \orgaddress{\street{Lindtstedtsv\"agen 25}, \city{Stockholm}, \postcode{114 28}, \country{Sweden}}}

\affil[2]{\orgdiv{Institute for Mathematics}, \orgname{University of Rostock}, \orgaddress{\street{Universit\"atsplatz 1}, \city{Rostock}, \postcode{18051 Rostock}, \country{Germany}}}

\affil[3]{\orgdiv{Abteilung Mathematik}, \orgname{Universit\"at zu K\"oln}, \orgaddress{\street{Weyertal 86-90}, \city{K\"oln}, \postcode{50931}, \country{Germany}}}


\abstract{Finding point configurations, that yield the maximum polarization (Chebyshev constant) is gaining interest in the field of geometric optimization. In the present article, we study the problem of unconstrained maximum polarization on compact sets. In particular, we discuss necessary conditions for local optimality, such as that a locally optimal configuration is always contained in the convex hull of the respective darkest points. Building on this, we propose two sequences of mixed-integer linear programs in order to compute lower and upper bounds on the maximal polarization, where the lower bound is constructive. Moreover, we prove the convergence of these sequences towards the maximal polarization. }

\keywords{maximal polarization, potentials, mixed integer programming, geometric optimization}


\pacs[MSC Classification]{31C20, 51-08, 90C11}

\maketitle

\section{Introduction}

Suppose you were given a set $A$ and $N$ lamps you are to place such that the darkest point in $A$ is as bright as possible. In less descriptive terms this max-min problem is known as the maximal polarization problem, which we now state in mathematical language.

Let $A, D\subset\R^n$ be nonempty sets and let $K:A\times D\rightarrow\R\cup\{+\infty\}$ be a function bounded from below.
An $N$-point multiset $C\subseteq D$ will be refered to as \emph{point configuration (of $N$ points)} and the set of all $N$-point configurations supported on $D$ will be denoted by $\mathcal{C}$.
We assign the discrete $K$-potential associated with $C$ to every point $p\in A$ as
$$U_{K, A}(p, C) = \sum_{c\in C} K(p, c).$$
To any point configuration we associate its \emph{polarization}
$$P_{K, A}(C) = \inf_{p\in A} U_{K, A}(p, C).$$
It is then natural to consider the \emph{(maximal) polarization problem}:
\begin{equation}
	\label{eq:potential_polarization}
	\mathcal{P}_K(A) = \sup_{C\in\mathcal{C}} P_{K, A}(C).
\end{equation}

For a broader context and overview of this formulation of the polarization problem we refer to the recent monograph \cite[CH. $14$]{borodachov2019discrete}. Problems of this kind have been extensively studied. In particular the case of $A = D = S^{n-1}$ being a unit sphere and $K(x,y) = \|x-y\|^{-s}$ being related to a Riesz potential is rich in results on explicit optimal configurations of few points (eg. \cite{stolarsky1975sum},
\cite{ambrus2009analytic}, \cite{ambrus2010chebyshev},
\cite{nikolov2011sum}, \cite{hardin2013polarization},  \cite{erdelyi2013riesz} \cite{Borodachov2022a}), bounds on maximal polarization (eg. \cite{ambrus2010chebyshev}, \cite{erdelyi2013riesz}) and asymptotic results (eg. \cite{borodachov2014asymptotics}, \cite{borodachov2018optimal}, \cite{hardin2020unconstrained}, \cite{anderson2022polarization}). 
Asymptotic results are also available for more general choices of $A$, such as rectifiable sets.
\bigskip
Moreover, the polarization problem as stated in \eqref{eq:potential_polarization} is closely related to the well-studied covering problem, i.e. the question, whether $A$ can be covered by balls of radius $r>0$. In particular, let $K(x,y)=\1_{[0,r]}(\|x-y\|)$, then, a covering with $N$ balls exists if and only if $1\leq \mathcal{P}_K(A)$. General discussions of covering problems can be found, for example in the seminal book by Conway and Sloane \cite{conway2013sphere}. For covering problems on compact metric spaces we refer to \cite{Naszodi2018a} for an overview, whereas constructive methods have been developed, e.g. in \cite{Naszodi2014a} and \cite{Rolfes2017a}.

\bigskip

In this paper we consider polarization problems of the following kind.
The set $A\subset \R^n$ will be a compact set and we will impose no restrictions on the point configurations, i.e. $D = \R^n$. Furthermore, we restrict to functions $K(x, y) = f(\|x-y\|)$ for some continuous strictly monotone decreasing function $f : \R_+ \rightarrow\R_+$ and use the notation $U_{f, A}(p, C)$, $P_{f, A}(C), \mathcal{P}_f(A)$. If the subscript parameters are clear from context we omit them.

Under the above assumptions, we therefore consider the optimization problem 

\begin{equation}
	\label{eq:polarization}
	\mathcal{P}_{f}(A) = \sup_{C\subset\mathcal{C}} P_{f, A}(C).
\end{equation}

For explicit computations we choose Gaussians $f(x) = e^{-ax^2}$. These functions appear rather naturally in the context of universal optimality (cf. \cite{cohn2007universally}): Recall that a function $g:(0,\infty) \rightarrow \R$ is \emph{completely monotonic} if it is infinitely differentiable and the derivatives satisfy $(-1)^k g^{(k)} \geq 0$ for all $k$. The functions $g(x) = e^{-\alpha x}$ are completely monotonic and we can write $f(\|x-y\|) = g(\|x-y\|^2)$. In this context functions $f(x) = g(x^2)$ are called completely monotonic functions of squared distance.

A Theorem of Bernstein (cf. \cite[Thm. 9.16]{simon2011convexity}) asserts that every completely monotonic function can be written as a convergent integral
\[
g(x) = \int e^{-\alpha x} d\mu(\alpha).
\]
From this one obtains that the set of completely monotonic functions of squared distance is the cone spanned by the gaussians and the constant function $x \mapsto 1$.

In particular the commonly used Riesz potentials can be written in this way.

\bigskip

We fix some more notation for the case that the infimum $P_{f,A}$ is in fact a minimum, i.e. the minimizers of this function are points in $A$. In this case, any such minimizer will be called a \emph{darkest point} of $A$.
Moreover,
$$\dark_A(C) = \{p\in A \ : \ \sum_{c\in C} f(\|p- c\|) = P_{f, A}(C)\}$$
will be called the \emph{set of darkest points} of $C$. To explain this wording we invite the reader to recall the interpretation of the problem we gave in the beginning: we center lamps at the points in $C$ which now illuminate $A$. The polarization of $A$ is then the lowest level of brightness any point in $A$ can have, any point realizing this is a ``darkest point''.

Note, that requiring $A$ to be compact is rather natural. Indeed if $A$ were unbounded, then the value of the polarization would always tend to $N\cdot \inf f$.
If $A$ were not closed, darkest points need not exist. Consider for example $A$ to be the open disc and $C$ only containing the origin. In this case, $P_{f, A}(C)$ is not attained at any point in $A$.
\bigskip
In Section \ref{sec:darkest_points} we provide some results connecting a locally optimal configuration to the set of its respective darkest points. Theorem \ref{thm:necessary_condition} states that the points of such a configuration are contained in the convex hull of the darkest points while on the other hand Theorem \ref{thm:darkest:points} states that the darkest points are located either on the boundary of $A$ or in the interior of the convex hull of the configuration. These restrictions provide necessary conditions for optimality. 

In Section \ref{sec:mips} we investigate mixed-integer approximations of the polarization problem providing upper and lower bounds. These are collected in Theorem~\ref{thm:hierarchy}. We then prove that these bounds indeed converge to $\mathcal{P}_{f}(A)$ in Theorems \ref{thm:lower_convergence} and \ref{thm:upper_convergence}. 

In Section \ref{sec:comp_results} we illustrate capabilities and limitations of the approach on some benchmark instances.

\section{Darkest points and necessary conditions}
\label{sec:darkest_points}

In this section, we investigate structural properties an optimal configuration needs to satisfy in order to potentially falsify the optimality of a given polarization and reduce the search space of optimal configurations.

In particular, we have the following necessary condition that relates local optimality of a configuration to the set of its darkest points:

\begin{theorem}
	\label{thm:necessary_condition}
	If $C$ is a locally optimal solution of \eqref{eq:polarization}, then 
	$$C\subset \conv\dark_A(C).$$
\end{theorem}
\begin{proof}
	Suppose $C$ is a configuration for which we have $c\in C$ such that $c\notin \conv(\dark_A(C))$.
	In the following, we discuss how to construct a new configuration $C'$ in an arbitrary neighbourhood of $C$ such that $P(C') > P(C)$. Thus $C$ can not be locally optimal.
	Since $f$ is continuous, the niveau line
	$$S = \{p\in\R^n \ : \ U(p, C) = P(C)\}$$
	containing the darkest points is closed and thus $\dark_A(C) = A\cap S$ is compact.
	Therefore, $\conv(\dark_A(C))$ is a compact convex set and we can find a hyperplane $H = \{x \ : \ a^\top  x = b\}$ strictly separating this set from $c$ such that $a^\top  c < b$. 
	For $\eps > 0$ small enough $c' = c + \eps a$ still satisfies $a^\top c' < b$. We obtain a new configuration $C' = C\cup\{c'\}\setminus\{c\}$.
	Note, that for every neighbourhood of $C$, there is a sufficiently small $\eps$ such that $C'$ is contained in said neighbourhood.
	Obviously $|c' -p| < |c - p|$ for all points $p$ in the non-negative halfspace of $H$. In particular $c'$ is closer to all of the darkest points than $c$ and
	since $f$ is monotonously decreasing
	$$U(p, C') > U(p, C) \geq P(C)$$
	for all points $p$ in the non-negative halfspace of $H$.
	
	It remains to assert this also on the negative halfspace.
	Since all the darkest points are on the positive side of $H$, a point $p\in A\cap (H\cup H_{-})$ satisfies
	$$U(p, C) > P(C).$$
	Since $A\cap (H\cup H_{-})$ is compact this yields 
	$$U(p, C) \geq \delta > P(C)$$
	for some constant $\delta$.
	By continuity of $f$, for $\eps$ small enough, we can guarantee that
	$$U(p, C') > P(C)$$
	for all $p\in A\cap (H\cup H_{-})$.
	Altogether, 
	$$P(C') = \inf_{p\in A} U(p, C') > P(C).$$
\end{proof}

The formulated condition is very \enquote{unstable} in the following sense:
\begin{proposition}
	\label{prop:unstable}
	Let $C$ be a configuration such that $C\subset\conv\dark_A(C)$.
	Let $c\in C$ and $c'\neq c$ and $C' = C\cup\{c'\}\setminus\{c\}$.
	Then
	\begin{enumerate}
		\item $P(C') < P(C)$ and
		\item $C'\nsubseteq \conv\dark_A(C')$.
	\end{enumerate}
\end{proposition}

\begin{proof}\phantom{ }
	
	\begin{enumerate}
		\item Consider the hyperplane $H$ with outer normal $c - c'$ through $c$, oriented such that $c'$ is on the negative side. Since $c\in\conv\dark_A(C)$ there has to be a darkest point $d\in \dark_A(C)$ in the non-negative halfspace of $H$ (it might be in $H$). Then $\|c - d\| < \|c' -d\|$ and by monotonicity $f(\|c - d\|) > f(\|c' -d\|)$. The potentials $U(d,C')$ and $U(d,C)$ differ by $f(\|c'-d\|)-f(\|c-d\|)$, therefore the above implies 
		$$
		P(C') \leq U(d, C')
		< U(d, C) = P(C).
		$$
		\item Suppose $C'\subset\conv\dark_A(C')$. Then we can apply 1. to $C'$ with the roles of $c,c'$ reversed. But this would give $P(C') < P(C) < P(C'),$ which is a contradiction.
	\end{enumerate}
\end{proof}
Optimization methods which only consider single components (like pattern search) or move single configuration points therefore possibly converge to a configuration contained in the convex hull of the darkest points which is not locally optimal. Therefore it seems reasonable to only use optimization methods which are able to move several points at once. 
Another conclusion is the following, which seems to suggest that the number of optimization variables can be reduced to only $m-1$ vectors.
\begin{corollary}
	For given points $C'$ with $|C'| = N-1$ there is \emph{at most} one point $c$ such that $\{c\}\cup C'\subset\conv\dark_A(\{c\}\cup C')$.
\end{corollary}
We can use Theorem \ref{thm:necessary_condition} to study the structure of the darkest points even more. First, we discuss a way to find certificates for $p\notin\dark_A(C)$.
\begin{lemma}
	\label{lem:shadow}
	Let $C$ be a configuration and $p\in\R^n$ be an arbitrary point.
	Let  
	\[
	N(p,C) = \{p + v \ : \ v\neq 0 \text{ and } v^\top w \geq 0 \text{ for all } w\in \cone\{p-c \ : \ c\in C\}\}.
	\]
	\begin{enumerate}
		\item For all $q\in N(p, C)$ we have $U(q, C) < U(p, C)$,
		\item if $N(p, C) \cap A \neq \emptyset$ then $p\notin\dark_A(C)$.
	\end{enumerate}
\end{lemma}
\begin{proof}
	Write $q = p + v\in N(p, C)$ with $v\neq 0$.
	Then for all $c\in C$ we have
	$$|c - (p+v)|^2 = |c - p|^2 + 2(p-c)^\top  v + |v|^2 > |c-p|^2.$$
	Since $f$ is strictly monotone decreasing, we have $U(q, C) < U(p, C)$.
	From this, the second claim follows immediately.
\end{proof}

The above definition of $N(p,C)$ of a point $p$ 
contains only points at which the potential is strictly smaller than at $p$ itself, as we just showed. We think this object will be useful beyond the scope of the previous lemma and subsequent theorem, but in the present work we only need it here. 

If we recall the visualization of the polarization problem as placing light sources $C$ to illuminate $A$ the above definition of $N(p,C)$ of a point $p$ contains only points that are illuminated less than $p$ itself. It is (by cone duality) somewhat related to the idea of a physical shadow (which would be resembled most closely by $-\cone\{p-c:c \in C\}$). With this we prove the following result which further restricts the location of the darkest points:

\begin{theorem} \label{thm:darkest:points}
	Let $C$ be a feasible configuration for \eqref{eq:polarization}. Then the points of $\dark_A(C)$ are either in the interior of $\conv(C)$ or in $\delta A$, i.e. $\dark_A(C) \subset \inter\conv(C) \cup \delta A$.
	Moreover, if $C$ is locally optimal for \eqref{eq:polarization}, then $\dark_A(C)\cap \delta A \neq \emptyset$. 
\end{theorem}

\begin{proof}    
	Let $p \in \dark_A(C)$ and assume $p \notin \inter\conv(C)$. 
	Furthermore, let $N(p, C)$ be defined as in Lemma \ref{lem:shadow}. We can find a hyperplane $H = \{x \in \R^n\ :\ a^\top  x = \beta \}$ through $p$ separating $C$ from $p$, in particular $a^\top c \leq \beta$ for all $c\in C$. Then for all $c \in C$
	\[
	a^\top  (p-c) = a^\top p - a^\top c = \beta - a^\top c \geq 0,
	\]
	which shows that $p + \lambda a \in N(p,C)$ for arbitrary $\lambda >0$. If $p \in \inter A$, so is $p + \lambda a$ for $\lambda$ sufficiently small. Then $A \cap N(p,C) \neq \emptyset$ in contradiction to Lemma \ref{lem:shadow}. Thus $p \in \partial A$ as claimed.

	In addition, if $C$ is also locally optimal for \eqref{eq:polarization} by Theorem \ref{thm:necessary_condition} we immediately obtain that $C\subset\conv\dark_A(C)$. Now, assume $\dark_A(C)\cap \delta A = \emptyset$, then as seen above $\dark_A(C) \subseteq \inter \conv C$ and we obtain 
	$$C \subseteq \conv \dark_A(C) \subseteq \inter \conv C,$$
	which is a contradiction since $C$ is finite.
\end{proof}

\begin{figure}[H]
	\centering
	\includegraphics[scale=0.5, trim = 0 0 0 0.5cm, clip]{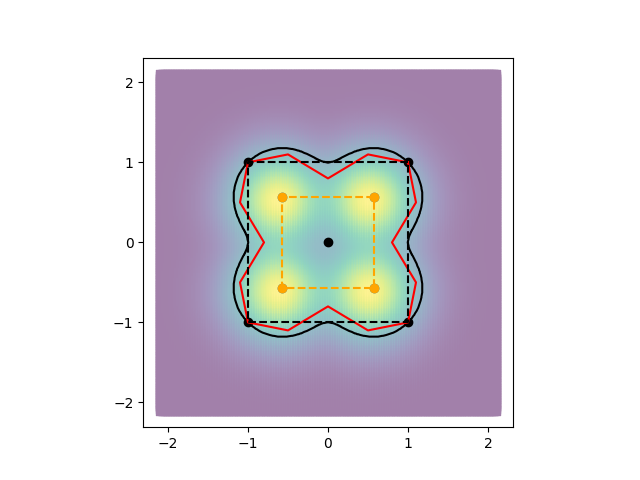}
	\caption{Illustration of Theorems \ref{thm:necessary_condition} and \ref{thm:darkest:points}. $A$ is depicted in red, $\dark_A$ in black and the configuration $C$ in orange. The dashed lines depict the convex hulls $\conv C$ and $\conv \dark_A(C)$, whereas the black line depicts all points $p\in \R^2$, such that $U(p,C)=P(C)$.}
	\label{fig:illustration_of_proofs}
\end{figure}

To summarize, locally optimal configurations $C$ of \eqref{eq:polarization} and its corresponding darkest points $\dark_A(C)$ share a similar containment property as is illustrated in Figure \ref{fig:illustration_of_proofs}.

\section{An MIP approach to polarization}
\label{sec:mips}
The current section is dedicated to the development of two hierarchies of mixed-integer linear programs (MIP) that approximate the maximal polarization of a compact set $A$ with respect to a monotonically decreasing and continuous function $f:\ \R_+ \rightarrow \R_+$. The MIP, that computes the lower bounds is constructive, i.e. solutions to this MIP are configurations whose polarization is lower bounded by the value of the MIP. The actual polarization of these configurations may very well exceed this lower bound by a significant margin, cf. Figure \ref{fig:convergence} for some numerical evidence.

\bigskip

First we give an equivalent description of problem \eqref{eq:polarization}. For this we observe that by Theorem \ref{thm:necessary_condition} any locally optimal point configuration is necessarily supported on $\conv(A)$. Furthermore we can get rid of the infimum by adding new constraints. The resulting optimization problem is then

\begin{align}
	\mathcal{P}_{f}(A)= \max_{x,C}\ & x  && \label{prob:original}\\
	& C\in\multibinom{\conv(A)}{N} \notag \\
	& x \leq U_{f, A}(p, C) && \text{ for all } p\in A, \notag
\end{align}
where $\multibinom{X}{N}$ describes the set of all multisets 
of size $N$ with elements in $X$. It is now clear, that the $\sup$ is actually a $\max$, since the feasible region can easily be made compact by bounding $x$ from below (e.g. $x \geq 0$) without changing the value of the program. 
\bigskip

\subsection{MIP Hierarchies}\label{subsec:mips}
We observe that Problem \eqref{prob:original} is an optimization problem with finitely many variables (namely $x,C$), but infinitely many constraints - it is a semiinfinite program (SIP) - and therefore not solvable using standard solvers.
In the remainder of this section we introduce two hierarchies of (tractable) MIPs, that approximate $\mathcal{P}(A)$ from above and below (see Theorem \ref{thm:hierarchy}). For this we make use of the following concept of functions which \enquote{control} the difference of two values of $f$.

\begin{definition}
We call a family of functions $g_{c, p}:\R_+\rightarrow \R_+$ for $c\in \conv(A)$, $p \in A$ a \emph{family of control functions} (with respect to $f$, $A$) if for all $c \in \conv(A)$, $p\in A$:
\begin{enumerate}
	\item $g_{c, p}(0) = 0$,
	\item $g_{c, p}$ is continuous and non-decreasing,
	\item $|f(\|c - p\|) - f(\|c' - p\|)| \leq g_{c, p}(\|c-c'\|)$ for all $c'\in \conv(A)$,
	\item $|f(\|c- p\|) - f(\|c- p'\|)| \leq g_{c, p}(\|p-p'\|)$ for all $p\in A$,
\end{enumerate}
where $\|\cdot\|$ denotes the standard Euclidean norm.
\end{definition}

Note that $f$ is related to a function $K$ taking two points $c,p$ as arguments: $K(c,p) = f(\|c-p\|)$. A family of control functions allows us to control the way $K$ changes as we vary either $c$ or $p$.

This control will be an important ingredient of the proof of Theorem \ref{thm:hierarchy}. For continuous functions this is related to bounding the slope of $K$ as can be illustrated by the following example:
Suppose the function $K(c, \cdot) = f(\|c - \cdot\|)$ is Lipschitz-continous with Lipschitz constant $L$ for all $p\in A$. Then, $g_{c, p}(\eps) = L\cdot \eps$ is a valid control function for $f$.

However, applying global Lipschitz-continuity is not a very precise approximation as it ignores local information around specific points $c,p$. Therefore we provide a more suitable family of control functions. 

\begin{proposition}
For $f$ monotonously decreasing and continuous the following is a family of control functions:
$$g_{c, p}(\varepsilon) = \max (\hat g_{c, p}(\varepsilon), \hat g_{c, p}(-\varepsilon))$$
where
$$\hat g_{c, p}(x) = \begin{cases}f(0) - f(\|c - p\|) \text{ if } x < - \|c- p\|\\ \left|f(\|c-p\| + x) - f(\|c-p\|) \right|\text{ otherwise}\end{cases}$$
\end{proposition}

\begin{proof}
We fix $c, p$ and write $g = g_{c, p}$ and $\hat g = \hat g_{c, p}.$
Clearly $g(0) = \hat g(0) = 0$. Since $f$ is continuous, so is $g$.

For $x\in (-\infty, -\|c-p\|)$ the function $\hat g(x)$ is constant. For $x\in (-\|c-p\|, 0)$ we have
$$\hat g(x) = f(\|c-p\| + x) - f(\|c-p\|)$$
which is decreasing since $f$ is decreasing.
For $x\in (0, \infty)$ we have
$$\hat g(x) = f(\| c-p\|) - f(\|c-p\| + x)$$
which is increasing since $f$ is decreasing.
Overall $g(\eps) = \max(\hat g(\eps), \hat g(-\eps))$ is an increasing function on $\R_+$.

By symmetry, it is sufficient to prove that $g$ provides an upper bound for $\Delta = |f(\|c-p\|) - f(\|c-p'\|)|$ for all $p'\in \conv(A)$. To this end, we use the triangle inequalities 
$$\|c-p\| - \|p'-p\| \leq \|c-p'\| \leq \|c-p\| + \|p' - p\|$$ 
and that $f$ is a decreasing function. Then, on the one hand if $\|c-p\| \leq \|c-p'\|$, we have
\begin{align*}
	\Delta &= f(\|c-p\|) - f(\|c-p'\|) \\
	&\leq f(\|c-p\|) - f(\|c-p\| + \|p'-p\|) = \hat g(\|p'-p\|) \leq g(\|p'-p\|).
\end{align*}
On the other hand, if $\|c-p\| \geq \|c-p'\|$, we obtain
\begin{align*}
	\Delta &= - f(\|c-p\|) + f(\|c-p'\|) \\
	&\leq - f(\|c-p\|) + f(\|c-p\| - \|p-p'\|) = \hat g(-\|p-p'\|) \leq g(\|p-p'\|).
\end{align*}
\end{proof}

For explicit computations we need to discretize two aspects of the problem. Firstly, we discretize the set of possible point configurations. For this we choose a finite sample $\Lambda \subset \conv(A)$ and only optimize over
\begin{equation} \label{eq:sample:lights}
C \in \multibinom{\Lambda}{N}.
\end{equation}
Secondly, we replace the infinite number of constraints, parameterized by $A$, by a finite subcollection. For this we again choose a finite sample $\Gamma \subset A$, and only consider the inequalities
\begin{equation} \label{eq:sample:constraints}
x \leq U(p,C) \text{ for all } p \in \Gamma.
\end{equation}
However, this naively sampled problem is not necessarily connected to the original problem, since we enforce only a subset of the infinitely many constraints and allow only a finite number of configurations. Either one of these changes would provide valid bounds but they unfortunately work in different directions. We will now show how to overcome this problem by utilizing the above family of control functions to obtain lower and upper bounds on the original problem.



Let us first consider lower bounds on \eqref{prob:original}. It is clear that we can restrict the choice of configurations to be supported on a finite sample $\Lambda$ of $\conv(A)$ as in \eqref{eq:sample:lights} and obtain a program that computes a lower bound. 

Discretizing the constraints is the harder part, since removing constraints lets the maximum grow. The following lemma shows how a slight variation of discretized constraints for some fintie sample $\Gamma$ of $A$ imply the validity of all of the infinitely many original constraints.
\begin{lemma}
\label{lemma:lower_bound}
Let $g_{c, p}$ be a family of control functions.
Let $\varepsilon > 0$, $\Lambda$ be an arbitrary finite sample of $\conv(A)$ and $\Gamma$ be an $\varepsilon$-net of $A$. Furthermore, suppose $x\in\R, C\subset\multibinom{\Lambda}{N}$  satisfy
$$x \leq \sum_{c\in\Lambda} \1_C(c) \cdot \left(f(\|c-p\|) - g_{c, p}(\varepsilon)\right) \text{ for all } p\in \Gamma.$$
Then,
$$x\leq \sum_{c\in\Lambda} \1_C(c)\cdot f(\|c-p\|) \text{ for all } p\in A.$$
\end{lemma}

\begin{proof}
Let $p\in A$ be arbitrary and $n(p)=\argmin_{\bar{p}\in \Gamma}\{\|p-\bar{p}\|\}$ denote the closest sample point to $p\in A$. Note that $\|p-n(p)\| < \varepsilon$ since $\Gamma$ is an $\varepsilon$-net.
Then,
\begin{align*}
	\sum_{c\in \Lambda} \1_C(c) f(\|c-p\|) & = \sum_{c\in \Lambda} \1_C(c)\cdot\left(f(\|c-p\|) - f(\|c-n(p)\|)\right)\\
	& \qquad +\sum_{c\in \Lambda} \1_C(c) \cdot \left(f(\|c-n(p)\|) - g_{c, n(p)}(\varepsilon)\right)\\
	& \qquad +\sum_{c\in\Lambda}\1_C(c)\cdot g_{c, n(p)}(\varepsilon)\\
	& \geq - \sum_{v\in \Lambda} \1_C(c)\cdot g_{c, n(p)}(\|p - n(p)\|)\\
	& \qquad + x\\
	& \qquad + \sum_{c\in\Lambda}\1_C(c)\cdot g_{c, n(p)}(\varepsilon),
\end{align*}
which is larger than $x$ since $g_{c,n(p)}$ is non-decreasing and $\|p-n(p)\| < \varepsilon$.
\end{proof}

Conversely, if we consider upper bounds on \eqref{prob:original}, we now cannot simply choose a finite sample $\Lambda$ of $\conv(A)$ to approximate the above SIP.
Indeed this would restrict the set of feasible solutions of \eqref{prob:original} and thereby lower the maximum instead. Again, the following lemma provides a way around this problem using a variation of the constraints.

\begin{lemma}
\label{lemma:upper_bound}
Let $g_{c, p}$ be a family of control functions.
Let $\eps > 0$ and $\Lambda$ be an $\eps$-net of $\conv(A)$.
Furthermore, suppose $C\in\multibinom{\conv(A)}{N}$ and $x$ satisfy
$$x \leq U(p, C) = \sum_{c\in C} f(\|c - p\|) \text{ for all }p\in \Gamma.$$
Then, there exists a configuration $C'\in\multibinom{\Lambda}{N}$ such that
$$x \leq \sum_{c\in C'} f(\|c - p\|) + g_{c, p}(\varepsilon) \text{ for all }p\in \Gamma.$$
\end{lemma}
\begin{proof}
Let $C' = \{\{n(c) \ : \ c\in C\}\}$ where $n(c) = \argmin_{c'\in\Lambda} \|c - c'\|$.
Then
\begin{align*}
	\sum_{c\in C'} f(\|c - p\|) + g_{c, p}(\varepsilon) &= \sum_{c\in C} f(\|n(c) - p\|) - f(\|c - p\|)\\
	&\qquad + \sum_{c\in C} f(\|c - p\|) + \sum_{c\in C} g_{n(c), p}(\varepsilon)\\
	&\geq -\sum_{c\in C} g_{n(c), p}(\|c - n(c)\|) + x + \sum_{c\in C} g_{n(c), p}(\varepsilon) \geq x,
\end{align*}
where the last inequality holds since $g_{n(c),p}$ is non-decreasing and $\|c-n(c)\|<\eps$ as $\Lambda$ is an $\eps$-net of $\conv(A)$.
\end{proof}

Now we can prove the main result of this section.

\begin{theorem}
\label{thm:hierarchy}
Let $\eps_\Lambda, \eps_\Gamma > 0$ and $\Lambda$ be an $\eps_\Lambda$-net of $\conv(A)$ and $\Gamma$ be an $\eps_\Gamma$-net of $A$.
Furthermore, let $g_{c, p}$ be a family of control functions. Then we have the following:
\begin{subequations}
	\begin{align}
		\max\ & x \label{prob:lower_bound}\\
		& y \in \{0,\ldots , N\}^\Lambda \notag \\
		& \1^\top  y = N \notag\\
		& x \leq \sum_{c\in\Lambda} y_c\cdot (f(\|c - p\|) - g_{c, p}(\eps_\Gamma)) && \text{ for all } p\in\Gamma \notag\\
		\leq \max\ & x \label{prob:lower_bound_inter}\\
		& y \in \{0,\ldots , N\}^\Lambda \notag\\
		& \1^\top  y = N \notag\\
		& x \leq \sum_{v\in \Lambda} y_c\cdot f(\|c - p\|) && \text{ for all } p\in A \notag\\
		\leq \mathcal{P}(A) \label{prob:pol*}\\
		\leq \max\ & x \label{prob:upper_bound_inter}\\
		& C \in\multibinom{\conv(A)}{N} \notag\\
		& x \leq \sum_{c\in C} f(\|c - p\|) &&\text{ for all } p\in \Gamma \notag\\
		\leq \max\ & x&\label{prob:upper_bound}\\
		& y\in\{0,\ldots , N\}^\Lambda \notag\\
		& \1^\top  y = N \notag\\
		& x \leq \sum_{c\in\Lambda} y_c\cdot (f(\|c - p\|) + g_{c, p}(\eps_\Lambda)) && \text{ for all } p\in\Gamma\notag
	\end{align}
\end{subequations}
\end{theorem}
\begin{proof}
We show, that feasible solutions of the left hand sides are also feasible for the right hand sides with the same objective value justifying the asserted inequalities.
First, observe that Lemma \ref{lemma:lower_bound} implies that a feasible solution $x,y$ of \eqref{prob:lower_bound} is also feasible for \eqref{prob:lower_bound_inter} and the objective values coincide.
Next, we consider a feasible solution $x, y$ of \eqref{prob:lower_bound_inter} and observe that $y$ encodes a multiset $C\in \multibinom{\Lambda}{N}\subseteq\multibinom{\conv(A)}{N}$. Moreover, $x, C$ satisfy the constraints in \eqref{eq:polarization} and with the same objective value $x$.
The next inequality follows rather immediately since \eqref{prob:upper_bound_inter} is a relaxation of \eqref{eq:polarization} due to dropping constraints for $p\in A\setminus \Gamma$.
Lastly, if $x, C$ is a feasible solution of \eqref{prob:upper_bound_inter}, we apply Lemma \ref{lemma:upper_bound} to obtain a set $C'\in\multibinom{\Lambda}{N}$ satisfying the constraints of \eqref{prob:upper_bound}. Then, by encoding $C'$ through $y\in\{0,\ldots , N\}^\Lambda$ with $\1^\top  y = N$ we obtain a feasible solution to \eqref{prob:upper_bound} with the same objectve value $x$.
\end{proof}

Let us briefly comment on the computational complexity of the mixed-integer programs \eqref{prob:lower_bound} and \eqref{prob:upper_bound}. It is worth noting, that mixed-integer linear programming usually refers to optimization problems that include binary variables, which run significantly faster. We would like to note that the integral variables $y\in \{0,\ldots ,N\}^\Lambda$ in both \eqref{prob:lower_bound} and \eqref{prob:upper_bound} can be replaced by $|\Lambda| \cdot \log (N)$ binary variables. 

Moreover, in the lower bound of Theorem \ref{thm:hierarchy} the vector $y$ can be chosen as $y\in\{0,1\}^\Lambda$, which still provides a (potentially worse) lower bound and reduces the number of binary variables significantly. Unfortunately, a similar simplification is not immediately possible for the upper bound. However, we introduce another concept which aims to reduce the computational complexity in a similar fashion in the upper bound case.

\begin{definition}
\label{def:k_order_eps_net}
A finite subset $\Lambda\subset\R^n$ is called an $(\eps, k)$-net of $A$ if
\begin{enumerate}
	\item $\Lambda\subset A$,
	\item For every $p\in A$ there are at least $k$ distinct points $p_1, \dots, p_k\in \Lambda$ such that $|p_i - p| < \eps$.
\end{enumerate}
\end{definition}

Using an $(\eps_\Lambda, N)$-net we obtain a hierarchy similar to Theorem \ref{thm:hierarchy} restricting the possible entries of $y$ to $\{0,1\}$.
\begin{proposition}
\label{prob:binary}
Let $\eps_\Lambda, \eps_\Gamma > 0$ and $\Lambda$ be an $(\eps_\Lambda, N)$-net of $\conv(A)$ and $\Gamma$ be an $\eps_\Gamma$-net of $A$. Furthermore, let $g_{c, p}$ be a family of control functions. Then,
\begin{align*}
	\eqref{prob:upper_bound_inter}
	\leq \max\ & x&\\
	& y\in\{0,1\}^\Lambda\\
	& \1^\top  y = N \\
	& x \leq \sum_{c\in\Lambda} y_c\cdot (f(\|c - p\|) + g_{c, p}(\eps_\Lambda)) && \text{ for all } p\in\Gamma
\end{align*}
\end{proposition}
\begin{proof}
The proof works similar to the proof of Theorem \ref{thm:hierarchy} by replacing $C' = \{\{n(c) \ : c\in C\}\}$ in the proof of Lemma \ref{lemma:upper_bound} by a set $C'$ of $N$ distinct points of $\Lambda$.
This is possible since $\Lambda$ is an $(\eps_\Lambda, N)$-net (see Definition \ref{def:k_order_eps_net}).
\end{proof}

A trivial example of an $(\eps, N)$-net can basically be obtained by a multiset consisting of $N$ copies of an $\eps$-net. However, in practise there are usually solutions that need fewer points, albeit more than a classical $\eps$-net.

\subsection{Convergence Results}
\label{subsec:convergence}

After establishing upper and lower bounds to $\mathcal{P}(A)$ through the hierarchies presented in Theorem \ref{thm:hierarchy}, we study the quality of these bounds. To this end, we show in this section, that solutions of the bounding problems \eqref{prob:lower_bound} and \eqref{prob:upper_bound} converge, as $\eps_\Lambda, \eps_\Gamma$ both tend to $0$, to a solution of the original problem \eqref{prob:original}. Both proofs rely in large parts on the proof of Lemma 6.1 in \cite{Shapiro2009a}, which proves similar convergence for more general semiinfinite programs, but include minor necessary modifications. At first, we focus on the lower bounds, i.e., we show, that \eqref{prob:lower_bound} converges to \eqref{prob:lower_bound_inter} as $\eps_\Gamma\rightarrow 0$:

\begin{theorem}
\label{thm:lower_convergence}
Let $(\eps_k)$ be a non-negative sequence converging towards $0$. Furthermore, for every $k\in \N$ choose an $\eps_k$ net $\Gamma_k$ of $A$.
Then, any accumulation point of a sequence $(x_k,y_k)_{k \in \N}$ of optimal solutions of \eqref{prob:lower_bound} w.r.t. $\Gamma_k$ and $\eps_k$ is an optimal solution of \eqref{prob:lower_bound_inter}.
\end{theorem}
\begin{proof}
Let $(\bar{x},\bar y)$ be an accumulation point of $(x_k,y_k)$. By passing to a subsequence we can assume that $(x_k,y_k) \rightarrow (\bar{x},\bar{y})$ if $k\rightarrow \infty$.
We are now going to prove, that $(\bar{x},\bar{y})$ is feasible and in fact optimal for \eqref{prob:lower_bound_inter}:

Consider an arbitrary $p\in A$ and observe that since $\Gamma_k$ is an $\eps_k$-net of $A$, there exists a sequence $(p_k)$ with $p_k\in \Gamma_k$ such that $p_k\rightarrow p$ as $k\rightarrow \infty$. We observe further, that for all $k$ we have
$$x_k \leq \sum_{c\in\Lambda} (y_k)_c \cdot (f(\|c - p_k\|) - g_{c, p_k}(\eps_k)) \leq \sum_{c\in\Lambda} (y_k)_c\cdot f(\|c - p_k\|)$$
and by taking limits
$$\bar x \leq \sum_{c\in \Lambda}\bar y_c \cdot f(\|c - p\|).$$
Hence, $(\bar{x},\bar{y})$ is feasible for \eqref{prob:lower_bound_inter}.

Now, let $(x,y)$ be an arbitrary solution to \eqref{prob:lower_bound_inter}.
Since $A$ is compact and $\eps_k>0$, we know that $g_c = \max_{p\in A} g_{c, p}$ is a continuous, monotonously non-decreasing function with $g_c(0) = 0$. We now observe, that 
$$(x- \sum_{c\in\Lambda} y_c\cdot g_c(\eps_k), y)$$ is feasible for \eqref{prob:lower_bound} with respect to $\Gamma_k$. Since $(x_k, y_k)$ is an optimal solution to \eqref{prob:lower_bound}, we have $x_k\geq x  - \sum_{c\in\Lambda} y_c\cdot g_c(\eps_k)$. Consequently, as  $g_c(0)=0$, in the limit we obtain that $\bar{x} \geq x$. Since $x$ was chosen arbitrarily, we conclude, that $(\bar{x},\bar{y})$ is indeed optimal for \eqref{prob:lower_bound_inter}.
\end{proof}

Note that the convergence of \eqref{prob:lower_bound_inter} to \eqref{prob:pol*} as $\eps_\Lambda\rightarrow 0$ follows directly since the utility function and $f$ are continuous. Thus, Theorem \ref{thm:lower_convergence} implies the convergence of \eqref{prob:lower_bound} to \eqref{prob:pol*}, i.e. the value of \eqref{prob:lower_bound} tends to $\mathcal{P}(A)$, as $\eps_\Lambda,\eps_\Gamma\rightarrow 0$.

Moreover, with the same arguments, we conclude the convergence of \eqref{prob:upper_bound_inter} to \eqref{prob:pol*} as $\eps_\Gamma\rightarrow 0$ and thus only one proof of convergence remains, namely that \eqref{prob:upper_bound} converges to \eqref{prob:upper_bound_inter} as $\eps_\Lambda\rightarrow 0$.

One difficulty of the following theorem is the different kinds of feasible solutions when altering the sample $\Lambda$.
Feasible solutions of \eqref{prob:upper_bound} have the form $y\in \{1, \dots, N\}^\Lambda$ with $\1^\top y = N$ while feasible solutions of \eqref{prob:upper_bound_inter} are $N$-point multisets supported on $\conv(A)$.
Note that these objects do not permit an easy discussion of convergence.
However, both notions can be translated into an element $\omega\in(\conv(A))^N$ which is independent of $\Lambda$ and allows a discussion of convergence.
Note that $\omega$ can canonically be translated back into a multiset.

\begin{theorem}
\label{thm:upper_convergence}
Let $(\eps_k)$ be a non-negative sequence converging towards $0$. Furthermore, for every $k\in \N$ choose an $\eps_k$-net $\Lambda_k$ of $\conv(A)$.
Let $(x_k,y_k)$ be a sequence of optimal solutions of \eqref{prob:upper_bound} w.r.t. $\Lambda_k$, $\eps_k$.
Identifying each $y_k$ with $\omega_k\in(\conv(A))^N$, any accumulation point $(\bar x,\bar \omega )$ of this sequence corresponds to an optimal solution of \eqref{prob:upper_bound_inter} by identification of $\bar\omega$ with a multiset.  
\end{theorem}

\begin{proof}
The proof is similiar to the proof of Theorem \ref{thm:lower_convergence}.
Note that, since order of elements is not important for the discussed problems, we can regard to elements of $(\conv(A))^N$ either as tuples or as multisets depending on the context. Suppose $(x_k, \omega_k)$ with has an accumulation point $(\bar x, \bar \omega)$. By passing to a subsequence we can assume that $(x_k, \omega_k)\rightarrow (\bar x, \bar \omega)$. Consider the continuous function $g_p = \max_{c\in\conv(A)} g_{c, p}$ with $g_p(0) = 0$. Then, we have for all $k$ and $p\in\Gamma$:
\begin{align*}
	x_k &\leq \sum_{c\in\Lambda_k} (y_k)_c\cdot(f(\|c - p\|) + g_{c, p}(\eps_k))\\
	&\leq \sum_{i=1}^N f(\|(\omega_k)_i - p\|) + g_p(\eps_k)
\end{align*}
By taking limits we obtain
$$\bar x \leq \sum_{i=1}^N f(\| \bar\omega_i - p\|)$$
for all $p\in\Gamma$.
Thus $\bar x, \bar \omega$ is feasible for \eqref{prob:upper_bound_inter}.

Now suppose $x, \omega$ is an arbitrary solution of \eqref{prob:upper_bound_inter}.
Then by Lemma \ref{lemma:upper_bound} there exists $\omega_k'$ such that $x, \omega_k'$ is a feasible solution for \eqref{prob:upper_bound_inter}.
Since $(x_k, \omega_k)$ is an optimal solution, we have
$x_k \geq x$ and by taking limits
$\bar x \geq x$.
Therefore $\bar x$ is also optimal for \eqref{prob:upper_bound}.
\end{proof}

Note, that the proofs of Theorems \ref{thm:lower_convergence}, \ref{thm:upper_convergence} still work if we restrict $y$ to be binary as was discussed at the end of Section \ref{subsec:mips}.

Combining Theorems \ref{thm:lower_convergence} and \ref{thm:upper_convergence}, we conclude that by choosing a suitable sequence $(\eps_\Gamma)_k, (\eps_\Lambda)_k$, we can in theory bound the value of $\mathcal{P}(A)$ as tightly as we need. However, solving the respective mixed-integer linear problems in practice will pose a computational challenge.

\section{Computational results}
\label{sec:comp_results}

This section presents numerical experiments  illustrating the capabilities and limits of the MIP approach presented in this paper.
All computations have been performed using Gurobi on a HP DL380 Gen9 server with two Intel(R) Xeon(R) CPU E5-2660v@2.00GHz (each with 14 cores) and 256 GB RAM.
We first focus on a simple illustrative example, where $A$ is an equilateral triangle and the size of the configuration is $N=3$. In addition, we chose $f(x) = e^{-5\|x\|^2}$ for our potential function and $\varepsilon_\Gamma=0.014, \varepsilon_\Lambda = \varepsilon_\Gamma /3$ as the respective discretization widths of $\Gamma\subseteq A$ and $\Lambda\subseteq \conv(A)$. Lastly, we restrict both, \eqref{prob:lower_bound} and \eqref{prob:upper_bound} to binary variables $y\in \{0,1\}^\Lambda$ instead of integral $y\in \{0,\ldots , N\}^\Lambda$ as discussed below Theorem \ref{thm:hierarchy}. Since we expect the resulting configuration to consist of three separate points, this should not significantly impact the quality of the bounds. 

We illustrate the configuration given by \eqref{prob:lower_bound} in Figure \ref{fig:lower_bound_configuration}. It was obtained after approximately $10$ hours.
\begin{figure}[H]
\centering
\includegraphics[scale=0.5, trim = 0 0 0 1cm]{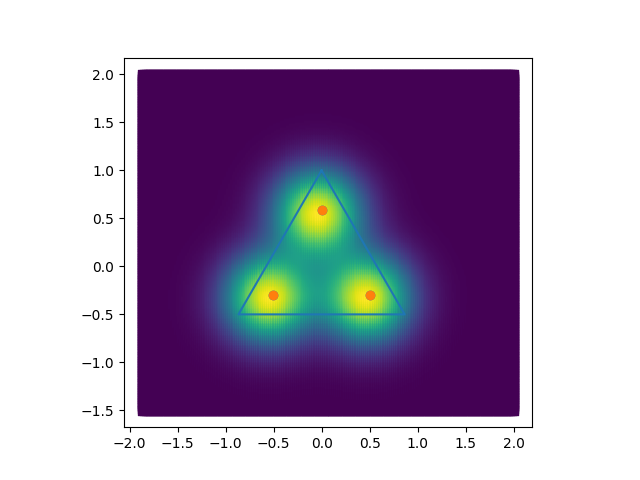}
\caption{Optimal configuration for \eqref{prob:lower_bound} with $\eps = 0.014$ and a heatmap of the respective $f$-potential (from dark blue over green to yellow). The points of the configuration are represented by orange circles.}
\label{fig:lower_bound_configuration}
\end{figure}
We continue by assessing the numerical evidence on the convergence for the above example. To this end, we illustrate the quality of the binary versions of both, \eqref{prob:lower_bound} and \eqref{prob:upper_bound} for decreasing values of $\eps_\Lambda$ and $\eps_\Gamma$. Here, the binary variant of \eqref{prob:upper_bound} was derived from Proposition \ref{prob:binary}. To be precise, for every $\eps\in\{0.04, 0.038, \dots, 0.014\}$ we computed the lower bound using $\eps_\Lambda =\eps/3$, $\eps_\Gamma = \eps$ and the upper bound using $\eps_\Gamma = \eps_\Lambda = \eps$. We chose these scalings for a better comparability, since the $(\eps_\Lambda, 3)$-net in the upper-bound case contains more sample points and therefore yields more variables than an $(\eps_\Lambda, 1)$-net.
Furthermore, we used scaled versions of the $A_2$ lattice complemented with additional sample points on the boundary to generate the samples $\Lambda$ and $\Gamma$. This construction ensures that both, $\Lambda$ and $\Gamma$ are indeed $\varepsilon_\Lambda$ and $\varepsilon_\Gamma$-nets respectively. The obtained bounds are visualized in Figure \ref{fig:convergence}.

It is apparent, that lower values of $\eps$ do not always yield better bounds although there is a clearly visible trend to close the gap between the bounds as can be expected from our convergence results established in Theorems \ref{thm:lower_convergence} and \ref{thm:upper_convergence}. A drawback of this approach is the computational runtime of the respective MIPs, which vastly increases with the sample size of $\Gamma$ and $\Lambda$ from a few seconds if $\eps=0.04$ to $10$ hours for $\eps=0.014$.

\begin{figure}[H]
\includegraphics[scale=0.3]{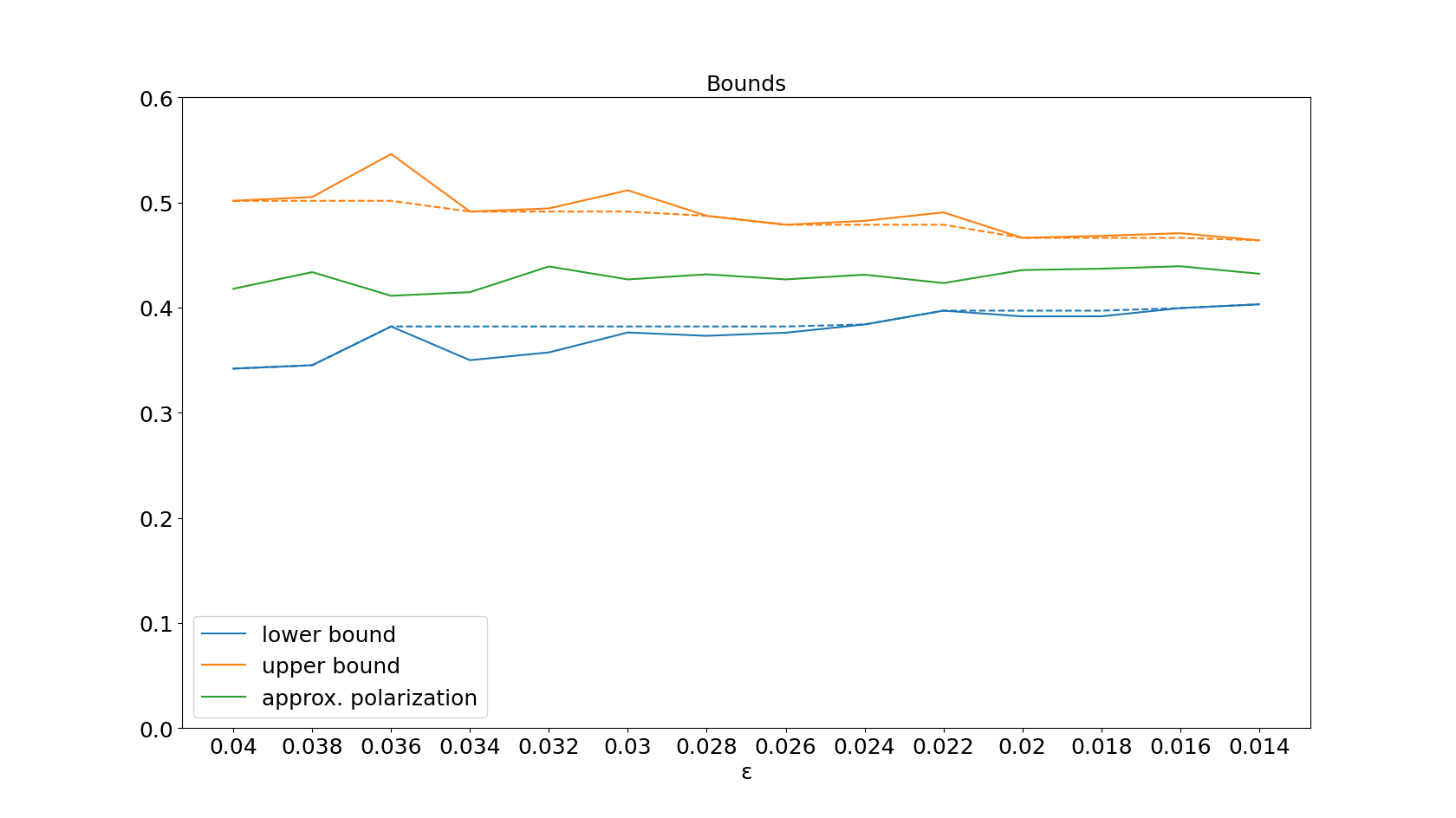}
\caption{Upper and lower bounds computed with decreasing values of $\eps$ and the respective running optimum (dashed lines) as well as an approximate polarization of the lower bound configuration.}
\label{fig:convergence}
\end{figure}

As an additional academic example, we use the same approach for different suitable choices of $\eps = \eps_\Lambda = \eps_\Gamma$ and different convex, non-convex or even non-connected $A$ to showcase the wide applicability of our approach. We illustrate the polarizations derived by the binary approximation of our lower bound MIP \eqref{prob:lower_bound} in Figure \ref{fig:computational_results}.
\begin{figure}[H]
\centering
\includegraphics[trim = 5cm 6cm 0 7cm, clip, scale=0.38]{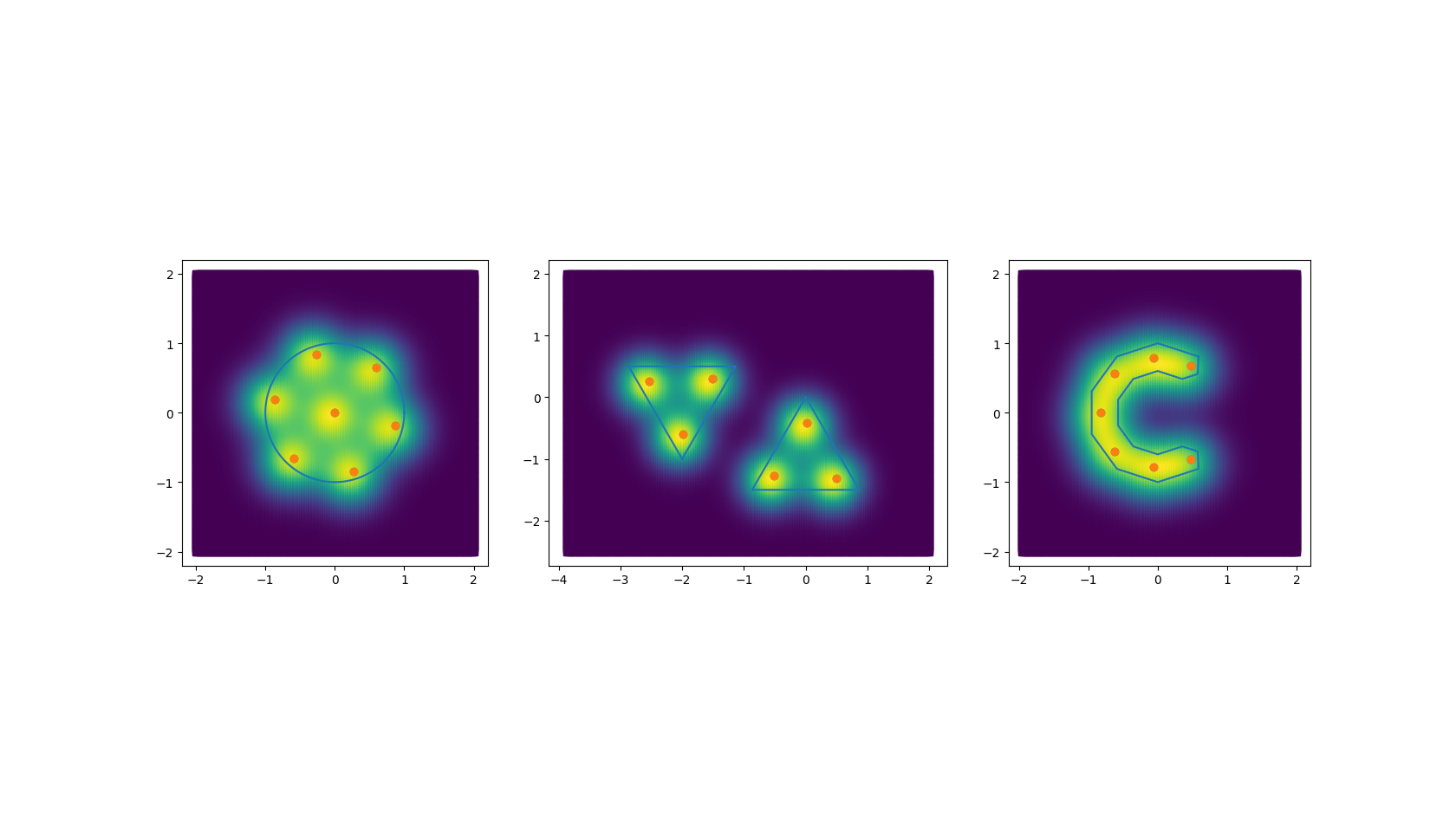}
\caption{
	Optimal configurations of \eqref{prob:lower_bound} for different $A$ in orange with a heatmap of the respective potential (from dark blue over green to yellow). The border of the respective shape $A$ is highlighted in blue (from left to right: ball, triangles, non-convex shape). 
}
\label{fig:computational_results}
\end{figure}
Moreover, we briefly summarize the computational results on these additional shapes $A$ in Table \ref{tab:computational_results} below. The respective sample widths were chosen such that the corresponding MIPs could be solved in reasonable time.
\begin{table}
\centering
\begin{tabular}{c|ccc}
	$A$ & Ball  & Two twisted triangles & Non-convex shape\\
	$N$ & 7 & 6 & 7\\
	\hline
	\hline
	lower bound & 0.391063 & 0.381283 & 0.918088\\
	\hline
	$\eps = \eps_\Lambda = \eps_\Gamma$ & 0.0625 & 0.025 & 0.025\\
	\begin{tabular}{c}computation time\\in seconds\end{tabular} & 6145 & 1238& 2020\\
	\hline
	\hline
	upper bound & 0.942982 & 0.506328 & 1.12719\\
	\hline
	$\eps = \eps_\Lambda = \eps_\Gamma$ & 0.0875 & 0.0375 & 0.03125\\
	\begin{tabular}{c}computation time\\in seconds\end{tabular} & 6741 & 858& 879\\
	\hline
	\hline
	gap  & ca. $59\%$ & ca. $25\%$ & ca $19\%$\\
\end{tabular}
\label{tab:computational_results}
\end{table}
We note, that the shape of $A$ significantly impacts the runtime of our MIP approach. It seems that the large symmetry group of the ball may contribute to a larger runtime as good solutions may be found everywhere in the branch-and-bound tree used by solvers such as Gurobi. If true, symmetry reduction techniques may lead to substantial improvements.

\section{Outlook}\label{sec:outlook}

We have seen in Section \ref{sec:darkest_points} that the location of the darkest points and the location of the points of a locally optimal configuration are intertwined. We suspect that these results can be extended, in particular by utilizing symmetries of $A$ or requiring $A$ to be convex or even a polytope. Furthermore, it would be interesting to extend these results to other choices of $D$.

However, it is clear that there will be limitations to these kinds of results. Consider for example $A = D = S^{n-1}$ the unit sphere. In this case, no obvious variant of Theorem \ref{thm:necessary_condition} holds.

In this paper, we have not dealt with explicit computations of locally or globally optimal point configurations, even on simple sets such as $n$-gons or the unit ball.
However, numerical experiments suggest that such configurations show some structure and we hope that extensions of the results in Section \ref{sec:darkest_points} can be utilized to obtain proof of optimality for some configurations.
Here, we would like to highlight one result in this direction we are aware of, namely that for certain Riesz potentials of modest decay and $A$ equals the closed $d$-dimensional unit ball, the optimal point configuration consists of $N$ copies of the origin (see \cite[Theorem 14.2.6]{borodachov2019discrete}). We were able to observe similar effects in numerical experiments on regular polytopes.

The MIP hierarchies presented in Section \ref{sec:mips} give provable upper and lower bounds converging to the optimal solution. However, unsurprisingly computing these bounds for sufficiently fine samples is very time consuming since MIP is NP-complete. A natural question is, whether well known techniques from mathematical programming - such as convex relaxations, inner approximations, column generation or local refinement, that speed up the computations can be utilized to achieve results for finer samples. However, most of these techniques only provide approximations of the discussed MIP hierarchies, which might limit the gain achieved through the finer samples. 

Moreover, it might be helpful to carefully fit the choice of the samples to the specific instance of the problem.
For example, if one has a conjecture for an optimal configuration and/or the correct location of the darkest points, this information can be fitted into the samples while retaining the $\varepsilon$-net property of the samples.
Furthermore, these ideas might provide a way to use our bounds for analytic proofs of optimality in highly structured situations.

\backmatter

\bmhead{Data Availability}
Data sharing not applicable to this article as no datasets were generated oranalysed during the current study.

\bmhead{Acknowledgments}

The authors like to thank Frank Vallentin for useful suggestions. M.C.Z. is partially supported ``Spectral bounds in extremal discrete geometry'' (project number 414898050) funded by the DFG.


\bibliography{sn-article}

\end{document}